\documentclass[
final
]{dmtcs-episciences}

\usepackage[utf8]{inputenc}
\usepackage{subfigure}

\usepackage{complexity}
\usepackage{tikz}
\usepackage{amsmath}

\newtheorem{theorem}{Theorem}[section]
\newtheorem{corollary}{Corollary}[section]
\newtheorem{lemma}{Lemma}[section]

\usepackage[round]{natbib}

\author[Mitre Dourado et. al]{Mitre C. Dourado\affiliationmark{1}\thanks{Partially supported by CNPq (305404/2020-2) and FAPERJ (211.753/2021), Brazil.}
  \and Vitor S. Ponciano\affiliationmark{1}\thanks{Partially supported by CAPES, Brazil.}
  \and R\^omulo L. O. da Silva\affiliationmark{2}}
\title[Corrigendum to ``On the monophonic rank of a graph"]{Corrigendum to ``On the monophonic rank of a graph" [Discrete Math. Theor. Comput. Sci. 24:2 (2022) \#3]}
\affiliation{
  Instituto de Computa\c{c}\~ao, Universidade Federal do Rio de Janeiro, Rio de Janeiro, Brazil\\
  Faculdade de Ci\^encia e Tecnologia, Instituto de Matem\'atica, Universidade Federal do Par\'a, Par\'a, Brazil}
\begin{document}
\publicationdata
{vol. 25:2}
{2024}
{26}
{10.46298/dmtcs.11423}
{2023-06-02; 2023-06-02; 2023-11-02}
{2023-11-23}
\maketitle
\begin{abstract}
In this corrigendum, we give a counterexample to Theorem 5.2 in ``On the monophonic rank of a graph" [Discrete Math. Theor. Comput. Sci. 24:2 (2022) \#3]. We also present a polynomial-time algorithm for computing the monophonic rank of a starlike graph.
\end{abstract}

\section{The counterexample}

In the paper ``On the monophonic rank of a graph", which appeared in \emph{Discrete Math. Theor. Comput. Sci.} 24:2 (2022) \#3, we claimed in Theorem~$5.2$ the \NP-completeness of {\sc Monophonic rank} for $2$-starlike graphs. This result was used in Corollary~$5.1$ to prove the \NP-completeness of {\sc Monophonic rank} for $k$-starlike graphs for any fixed $k \ge 2$. However, the reduction given in Theorem~$5.2$ is not correct. We show in the sequel that the example given in Figure~\ref{fig:G} illustrates a counterexample. Consequently, Corollary~$5.1$ does not hold as well.

\begin{figure}
	\centering
	
	\begin{tikzpicture}[scale=0.65]
		
		\pgfsetlinewidth{1pt}
		
		\tikzset{vertex/.style={circle,  draw, minimum size=5pt, inner sep=1pt}}
		
		\tikzset{setU/.style={circle,  draw, minimum size=35pt, inner sep=1pt}}
		
		\tikzset{vertexset/.style={rectangle,  draw, dashed, minimum size=25pt, inner sep=1pt, minimum height=1.5cm, minimum width=2cm}}
		
		\def\x{11}
		\def\y{13}
		
		\draw (-11.5+\x,2.5+\y) node[above] {$G$};
		
		\node[vertex] (v1) at (-9+\x,3+\y) [label=left:$v_1$]{};
		\node[vertex] (v2) at (-8+\x,3+\y) [label=right:$v_2$]{};
		\node[vertex] (v3) at (-8.5+\x,2+\y) [label=left:$v_3$]{} edge (v1) edge (v2);
		\node[vertex] (v4) at (-8.5+\x,1+\y) [label=left:$v_4$]{} edge (v3);
		\node[vertex] (v5) at (-9+\x,0+\y) [label=left:$v_5$]{} edge (v4);
		\node[vertex] (v6) at (-8+\x,0+\y) [label=right:$v_6$]{} edge (v4) edge (v5);
		\node[vertex] (v7) at (-8.5+\x,-1+\y) [label=left:$v_7$]{} edge (v4) edge (v5) edge (v6);
		
		\draw (-6,8) node[above] {$G'$};
		\draw[][black] (-4,-15) rectangle (3cm, 9cm); 
		
		\def\H{-1}
		\def\V{1}
		\def\i{1}
		
		\node[setU]  (U1) at (\H ,\V + 6.5) [label=left:$U_\i$]{};
		\node[vertex]  (u115) at (\H +3,\V + 6.5) [label=left:$u^{15}_\i$]{};
		\node[vertex] (u11) at (\H +9,\V + 7) [label=right:$w^{13}_\i$]{} edge(u115);
		\node[vertex]  (u14) at (\H +9,\V + 6) [label=right:$w^{14}_\i$]{} edge(u115) edge(u11);
		\node[vertexset]  (W1) at (\H +10 ,\V + 6.5) [label=right:$W_\i$]{};
		
		\def\V{-2.5}
		\def\i{2}
		
		\node[setU]  (U2) at (\H ,\V + 6.5) [label=left:$U_\i$]{};
		\node[vertex]  (u115) at (\H +3,\V + 6.5) [label=left:$u^{15}_\i$]{};
		\node[vertex] (u11) at (\H +9,\V + 7) [label=right:$w^{13}_\i$]{} edge(u115);
		\node[vertex]  (u14) at (\H +9,\V + 6) [label=right:$w^{14}_\i$]{} edge(u115) edge(u11);
		\node[vertexset]  (W2) at (\H +10 ,\V + 6.5) [label=right:$W_\i$]{};
		
		\def\V{-6}
		\def\i{3}
		
		\node[setU]  (U3) at (\H ,\V + 6.5) [label=left:$U_\i$]{} edge(W1) edge(W2);
		\node[vertex]  (u115) at (\H +3,\V + 6.5) [label=left:$u^{15}_\i$]{};
		\node[vertex] (u11) at (\H +9,\V + 7) [label=right:$w^{13}_\i$]{} edge(u115);
		\node[vertex]  (u14) at (\H +9,\V + 6) [label=right:$w^{14}_\i$]{} edge(u115) edge(u11);
		\node[vertexset]  (W3) at (\H +10 ,\V + 6.5) [label=right:$W_\i$]{} edge(U1) edge(U2);
		
		\def\V{-9.5}
		\def\i{4}
		
		\node[setU]  (U4) at (\H ,\V + 6.5) [label=left:$U_\i$]{} edge(W3);
		\node[vertex]  (u115) at (\H +3,\V + 6.5) [label=left:$u^{15}_\i$]{};
		\node[vertex] (u11) at (\H +9,\V + 7) [label=right:$w^{13}_\i$]{} edge(u115);
		\node[vertex]  (u14) at (\H +9,\V + 6) [label=right:$w^{14}_\i$]{} edge(u115) edge(u11);
		\node[vertexset]  (W4) at (\H +10 ,\V + 6.5) [label=right:$W_\i$]{} edge(U3);
		
		\def\V{-13}
		\def\i{5}
		
		\node[setU]  (U5) at (\H ,\V + 6.5) [label=left:$U_\i$]{} edge(W4);
		\node[vertex]  (u115) at (\H +3,\V + 6.5) [label=left:$u^{15}_\i$]{};
		\node[vertex] (u11) at (\H +9,\V + 7) [label=right:$w^{13}_\i$]{} edge(u115);
		\node[vertex]  (u14) at (\H +9,\V + 6) [label=right:$w^{14}_\i$]{} edge(u115) edge(u11);
		\node[vertexset]  (W5) at (\H +10 ,\V + 6.5) [label=right:$W_\i$]{} edge(U4);
		
		\def\V{-16.5}
		\def\i{6}
		
		\node[setU]  (U6) at (\H ,\V + 6.5) [label=left:$U_\i$]{} edge(W4) edge(W5);
		\node[vertex]  (u115) at (\H +3,\V + 6.5) [label=left:$u^{15}_\i$]{};
		\node[vertex] (u11) at (\H +9,\V + 7) [label=right:$w^{13}_\i$]{} edge(u115);
		\node[vertex]  (u14) at (\H +9,\V + 6) [label=right:$w^{14}_\i$]{} edge(u115) edge(u11);
		\node[vertexset]  (W6) at (\H +10 ,\V + 6.5) [label=right:$W_\i$]{} edge(U4) edge(U5);
		
		\def\V{-20}
		\def\i{7}
		
		\node[setU]  (U7) at (\H ,\V + 6.5) [label=left:$U_\i$]{} edge[bend right=2](W4) edge(W5) edge(W6);
		\node[vertex]  (u115) at (\H +3,\V + 6.5) [label=left:$u^{15}_\i$]{};
		\node[vertex] (u11) at (\H +9,\V + 7) [label=right:$w^{13}_\i$]{} edge(u115);
		\node[vertex]  (u14) at (\H +9,\V + 6) [label=right:$w^{14}_\i$]{} edge(u115) edge(u11);
		\node[vertexset]  (W7) at (\H +10 ,\V + 6.5) [label=right:$W_\i$]{} edge[bend right=2](U4) edge(U5) edge(U6);
		
	\end{tikzpicture}
	
	\caption{A counterexample to the proof of Theorem~$5.2$ in~(\cite{DPS-2022-mono-rank}).}
	
	\label{fig:G}
	
\end{figure}
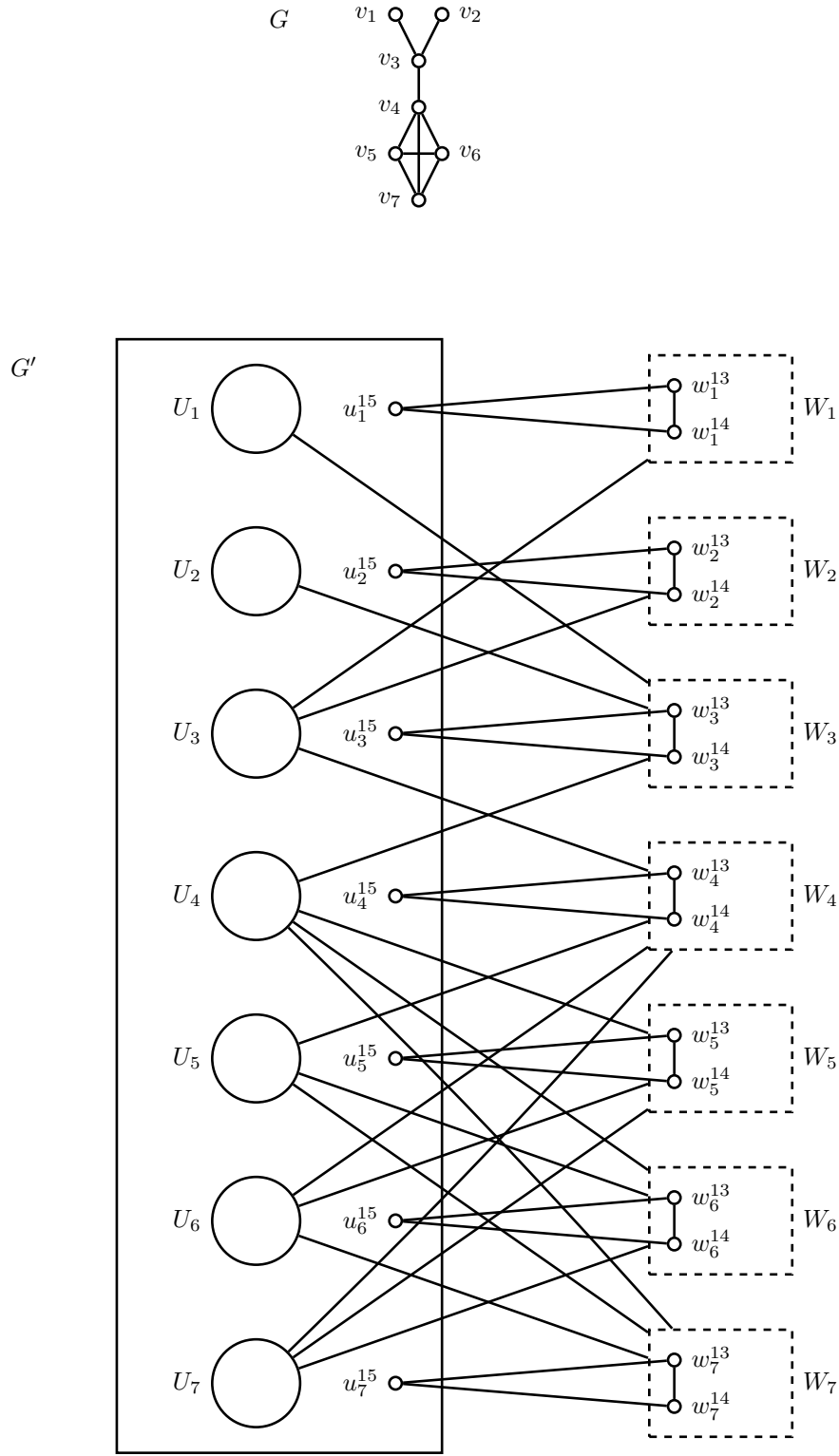

We need some definitions. Consider a graph $G$. The open and the closed neighborhoods of a vertex $v$ are denoted by $N(v)$ and $N[v]$, respectively. A vertex is {\em simplicial} if its closed neighborhood induces a complete graph. It is clear that a simplicial vertex is not an internal vertex of an induced path. We say that $S$ is {\em monophonically convex} if the vertices of every induced path joining two vertices of $S$ are contained in $S$. The {\em monophonic convex hull} of $S$, $\langle S \rangle$, is the smallest monophonically convex set containing $S$. A set $S$ is {\em monophonic convexly independent} if $v$ is not in $\langle S - \{v\} \rangle$ for $v \in S$. The {\em monophonic rank} of $G$, $r(G)$, is the size of a largest monophonic convexly independent set of $G$.

Recall that a {\em starlike graph} $G$ admits a partition of $V(G)$ into cliques $(V_0, V_1, \ldots, V_t)$ such that $V_0$ is a maximal clique and for $i \in \{1, \ldots, t\}$ and $u,v \in V_i$, it holds $N[u] - V_i = N[v] - V_i \subset V_0$. If $|V_\ell| \le k$ for $\ell \in \{1, \ldots, t\}$, then $G$ is \emph{$k$-starlike}~(\cite{Gustedt:1993,Cerioli2006}). The 1-starlike graphs are the {\em split graphs}.
Denote by $Z$ the set of simplicial vertices of $G$. For $i \in \{0, 1, \ldots, t\}$, we denote by $X_i$ the maximal clique containing $V_i$, $C_i = X_i \cap Z$ and $C'_i = X_i - C_i$. Note that $V_0 = X_0$, that $V_i = C_i$ for $i \in \{1, \ldots, t\}$, and that $C_0$ can be an empty set.

\begin{lemma} \label{lem:hull}
	If $G$ is a starlike graph and $S \subseteq V(G)$, then	every vertex $v \in \langle S \rangle - S$ belongs to $C'_0$ and is an internal vertex of an induced $(u,u')$-path such that $u \in S \cap Z \cap N(v)$ and $u' \in \langle S \rangle$.
\end{lemma}

\begin{proof}
	Let $v \in \langle S \rangle - S$. It is clear that $v$ is an internal vertex of an induced $(w,w')$-path $P$ for $w,w' \in \langle S \rangle$.
	Since $Z$ contains only simplicial vertices, we have that $v \in C'_0$.
	Since $C'_0$ is a clique, we have that at least one of $w$ and $w'$, say $w$, belongs to $Z$.
	Suppose first that $w'$ also belongs to $Z$. Note that $P$ has 3 or 4 vertices. In both cases, $v$ is adjacent to at least one of $w$ and $w'$.
	Suppose then that $w'$ belongs to $C'_0$. In this case, $P$ has $3$ vertices, which means that $v$ is adjacent to $w$.
	In all cases, $v$ is adjacent to some vertex of $Z$ belonging to $\langle S \rangle$.
	Since every vertex of $Z$ is simplicial, we have that such vertex belongs to $S$, completing the proof.
\end{proof}

Figure~\ref{fig:G} shows an input graph $G$ with $n = 7$ vertices and the resulting $2$-starlike graph $G'$ according to the reduction given in Theorem~$5.2$ in~(\cite{DPS-2022-mono-rank}), whose vertex set can be partitioned into sets $U$ and $W$ where $U$ is a clique with $70$ vertices and $G[W]$ has maximum clique of size $2$. In such proof, we claimed that $G$ has an independent set with $\lceil\frac{n+1}{2}\rceil = 4$ vertices if and only if $G'$ has a monophonic convexly independent set with $p$ vertices, where $p = n + (4n-1) \lceil \frac{n+1}{2} \rceil = 7 + 27 \times 4 = 115$.
It is easy to see that $G$ has no independent set of size $\lceil\frac{n+1}{2}\rceil = 4$. In order to see that the set $S \subset V(G')$ with $115$ vertices formed by the $56$ vertices of $U_1, U_2, W_1, W_2$, the $56$ vertices of $U_4, U_5, U_6, U_7$ and the $3$ vertices $u_5^{15}, u_6^{15}$ and $u_7^{15}$ is m-convexly independent
we use the notation given above where $V_0 = C'_0 = U$ and $Z = W$.
Since the vertices of $S \cap C'_0$ have no neighbors in $S \cap Z$, Lemma~\ref{lem:hull} implies that $S$ is indeed an m-convexly independent set of $G'$.

\section{Monophonic rank is polynomial for starlike graphs}

In Theorem~$5.1$ of~(\cite{DPS-2022-mono-rank}), we presented a polynomial-time algorithm for computing the monophonic rank of $1$-starlike graphs. Here, in Corollary~\ref{cor:starlike}, we extend such algorithm so that it works for starlike graphs.

Given a graph $G$ and a set $T \subseteq V(G)$, we write $N(T) = \underset{v \in T}{\cup}N(v)$.
If $T$ is an independent set, we define the \emph{difference} of $T$ as $d(T) = |T| - |N(T)|$,
and the \emph{critical independence difference} of $G$ as $d_c(G) = \max \{d(T) : T$ is an independent set of $G\}$. If $d(T) = d_c(G)$, then we say that $T$ is a critical independent set of $G$.

Using the notation given previously for a starlike graph $G$ with partition $(V_0, V_1, \ldots, V_t)$ of $V(G)$ into cliques, denote by $Y_i$ the union of the sets $C_j$ for $j \in \{0,1, \ldots, t\} - \{i\}$ such that $C'_j \subseteq C'_i$.
If $|N(Y_i) - Y_i| = |C'_i| - 1$ and $|Y_i| \ge |C'_i| -1$, then we can write $\{x\} = (N(Y_i) - Y_i) - C'_i$ and define $G_i = G[(C'_i - \{x\}) \cup Y_i] - E(G[Y_i])$, otherwise define $G_i = G[C'_i \cup Y_i] - E(G[Y_i])$.  Note that $G_i$ is a split graph where $Y_i$ is an independent set which can be empty. If this is the case, $G_i$ can be the empty graph. See an example in Figure~\ref{fig:splitP}.

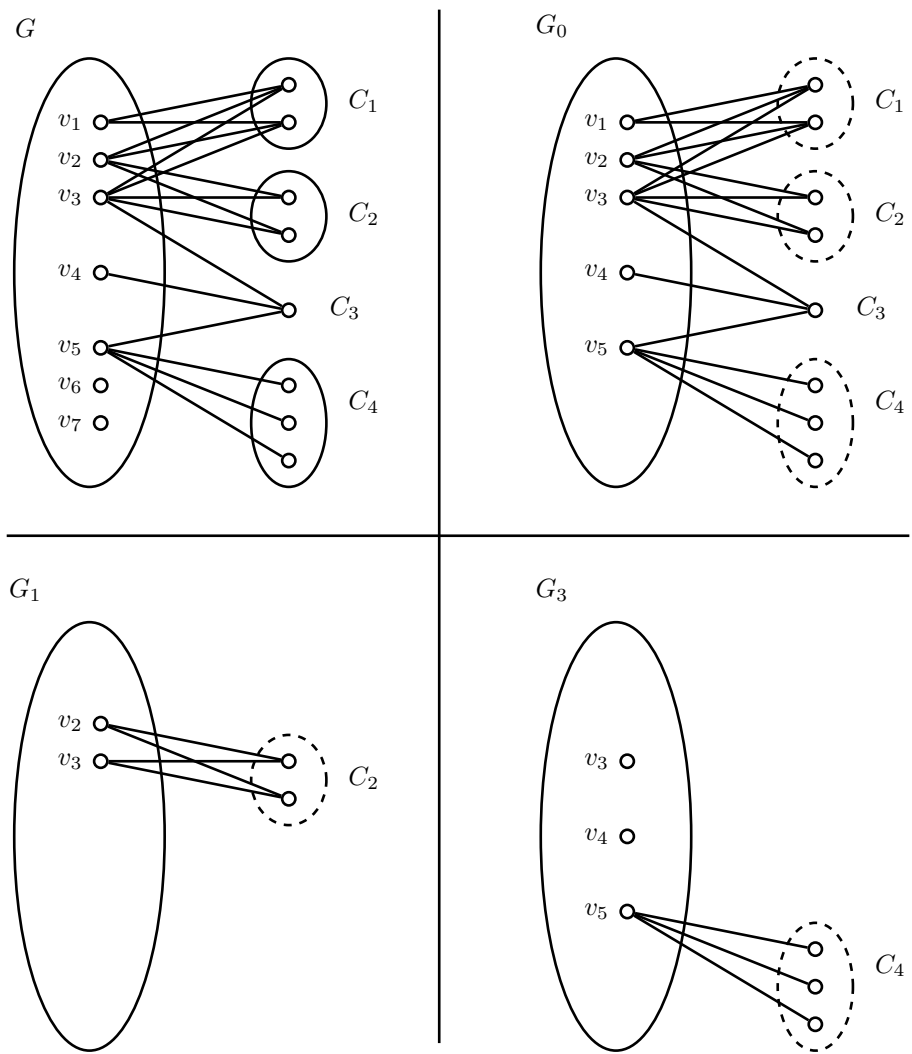
\begin{figure}
	\centering
	
	\begin{tikzpicture}[scale=0.5]
		
		\pgfsetlinewidth{1pt}
		
		\tikzset{vertex/.style={circle,  draw, minimum size=5pt, inner sep=0pt}}
		
		\draw (-6.5,-7) -- (17.5,-7);
		\draw (5,7) -- (5,-20.5);
		
		\begin{scope}[shift={(0,0)}]
			
			\draw (-6,6) node[above] {$G$};
			\draw[-] (-4.3, 0) ellipse (2cm and 5.7cm);
			\node[vertex] (v'1) at (-4,4) [label=left:$v_1$]{};
			\node[vertex]  (v'2) at (-4,3) [label=left:$v_2$]{} ;
			\node[vertex]  (v'3) at (-4,2) [label=left:$v_3$]{};
			\node[vertex] (v'4) at (-4,0) [label=left:$v_4$]{};
			\node[vertex] (v'5) at (-4,-2) [label=left:$v_5$]{};
			\node[vertex]  (v'6) at (-4,-3) [label=left:$v_6$]{} ;
			\node[vertex]  (v'7) at (-4,-4) [label=left:$v_7$]{} ;
			
			\draw (3,4) node[above] {$C_1$};
			\draw[-] (1, 4.5) ellipse (1cm and 1.2cm);
			\node[vertex] (w'1') at (1,5) [label=right:]{} edge (v'1) edge (v'2) edge (v'3);
			\node[vertex] (w'1'') at (1,4) [label=right:]{} edge (v'1) edge (v'2) edge (v'3);
			
			\draw (3,1) node[above] {$C_2$};
			\draw[-] (1, 1.5) ellipse (1cm and 1.2cm);
			\node[vertex] (w'2) at (1,2) [label=right:]{} edge (v'2) edge (v'3);
			\node[vertex] (w'3) at (1,1) [label=right:]{} edge (v'2) edge (v'3);
			
			\draw (2.5,-1.5) node[above] {$C_3$};
			\node[vertex] (w'6) at (1,-1) [label=right:]{} edge(v'3) edge(v'4) edge(v'5);
			
			\draw (3,-4) node[above] {$C_4$};
			\draw[-] (1, -4) ellipse (1cm and 1.7cm);
			\node[vertex] (w'4) at (1,-3) [label=right:]{} edge(v'5);
			\node[vertex] (w'4) at (1,-4) [label=right:]{}  edge(v'5);
			\node[vertex] (w'5) at (1,-5) [label=right:]{} edge (v'5);
			
		\end{scope}
		
		
		\begin{scope}[shift={(14,0)}]
			
			\draw (-6,6) node[above] {$G_0$};
			\draw[-] (-4.3, 0) ellipse (2cm and 5.7cm);
			\node[vertex] (v'1) at (-4,4) [label=left:$v_1$]{};
			\node[vertex]  (v'2) at (-4,3) [label=left:$v_2$]{} ;
			\node[vertex]  (v'3) at (-4,2) [label=left:$v_3$]{};
			\node[vertex] (v'4) at (-4,0) [label=left:$v_4$]{};
			\node[vertex] (v'5) at (-4,-2) [label=left:$v_5$]{};
			
			\draw (3,4) node[above] {$C_1$};
			\draw[dashed] (1, 4.5) ellipse (1cm and 1.2cm);
			\node[vertex] (w'1') at (1,5) [label=right:]{} edge (v'1) edge (v'2) edge (v'3);
			\node[vertex] (w'1'') at (1,4) [label=right:]{} edge (v'1) edge (v'2) edge (v'3);
			
			\draw (3,1) node[above] {$C_2$};
			\draw[dashed] (1, 1.5) ellipse (1cm and 1.2cm);
			\node[vertex] (w'2) at (1,2) [label=right:]{} edge (v'2) edge (v'3);
			\node[vertex] (w'3) at (1,1) [label=right:]{} edge (v'2) edge (v'3);
			
			\draw (2.5,-1.5) node[above] {$C_3$};
			\node[vertex] (w'6) at (1,-1) [label=right:]{} edge(v'3) edge(v'4) edge(v'5);
			
			\draw (3,-4) node[above] {$C_4$};
			\draw[dashed] (1, -4) ellipse (1cm and 1.7cm);
			\node[vertex] (w'4) at (1,-3) [label=right:]{} edge(v'5);
			\node[vertex] (w'4) at (1,-4) [label=right:]{}  edge(v'5);
			\node[vertex] (w'5) at (1,-5) [label=right:]{} edge (v'5);
			
		\end{scope}
		
		
		\begin{scope}[shift={(0,-15)}]
			
			\draw (-6,6) node[above] {$G_1$};
			\draw[-] (-4.3, 0) ellipse (2cm and 5.7cm);
			\node[vertex]  (v'2) at (-4,3) [label=left:$v_2$]{} ;
			\node[vertex]  (v'3) at (-4,2) [label=left:$v_3$]{};
			
			\draw (3,1) node[above] {$C_2$};
			\draw[dashed] (1, 1.5) ellipse (1cm and 1.2cm);
			\node[vertex] (w'2) at (1,2) [label=right:]{} edge (v'2) edge (v'3);
			\node[vertex] (w'3) at (1,1) [label=right:]{} edge (v'2) edge (v'3);
			
		\end{scope}	
		
		
		\begin{scope}[shift={(14,-15)}]
			
			\draw (-6,6) node[above] {$G_3$};
			\draw[-] (-4.3, 0) ellipse (2cm and 5.7cm);
			\node[vertex]  (v'3) at (-4,2) [label=left:$v_3$]{};
			\node[vertex] (v'4) at (-4,0) [label=left:$v_4$]{};
			\node[vertex] (v'5) at (-4,-2) [label=left:$v_5$]{};
			
			\draw (3,-4) node[above] {$C_4$};
			\draw[dashed] (1, -4) ellipse (1cm and 1.7cm);
			\node[vertex] (w'4) at (1,-3) [label=right:]{} edge(v'5);
			\node[vertex] (w'4) at (1,-4) [label=right:]{}  edge(v'5);
			\node[vertex] (w'5) at (1,-5) [label=right:]{} edge (v'5);
			
		\end{scope}
		
	\end{tikzpicture}
	
	\caption{Graphs $G_0, G_1$ and $G_3$ constructed from the starlike graph $G$. Ellipses formed by continuous lines represent cliques, while the ones formed by dashed lines represent independent sets.}
	
	\label{fig:splitP}	
\end{figure}

\begin{lemma} \label{lem:Gi}
	If $G$ is a starlike graph, then for $i \in \{0,1, \ldots, t\}$, $G_i$ has a critical independent set $T_i$ contained in $Y_i$.
\end{lemma}

\begin{proof}
	Note that $G_i$ is a split graph with bipartition $(C,Y_i)$ where $C = C'_i$ or $C = C'_i - \{x\}$ for $\{x\} = C'_i - N(Y_i)$. Let $T$ be an independent set of $G_i$ having a vertex $v \in C$. We show that there is an independent set $T'$ such that $T' \cap C \subset T \cap C$ such that $d(T') \ge d(T)$, which proves the result.
	
	If $v$ has a neighbor $y \in Y_i$, then $d((T - \{v\}) \cup \{y\}) \ge d(T)$ because $N[y] \subseteq N[v]$. Then, we can assume from now on that $v$ has no neighbors in $Y_i$. By the construction of $G_i$, we have that $C = C'_i$. If $|N(Y_i) - Y_i| < |C'_i| -1$, then $d(T - \{v\}) \ge d(T)$. It remains to consider $|N(Y_i) - Y_i| = |C'_i| -1$. In this case, we have that $|Y_i| < |C'_i| -1$. Finally, note that $d(T) \le 0$ and that $d(\emptyset) \ge d(T)$.
\end{proof}

\begin{theorem} \label{the:charac}
	If $G$ is a starlike graph, then $r(G) = \underset{i \in \{0,1, \ldots, t\}}{\max} \left\{ |X_i| + d_c(G_i) \right\}$.
\end{theorem}

\begin{proof}
	We begin by showing that for $i \in \{0,1, \ldots, t\}$, $G$ has an m-convexly independent set of size $|X_i| + d_c(G_i)$.
	Let $T_i$ be a critical independent set of $G_i$. By Lemma~\ref{lem:Gi}, we can assume that $T_i$ contains only vertices of $Y_i$.
	Define $S_i = (C'_i - N(T_i)) \cup T_i \cup C_i$.
	Observe that $|S_i| = |X_i| + d_c(G_i)$.
	Suppose by contradiction that $v \in \langle S_i - \{v\} \rangle$ for some $v \in S_i$.
	By Lemma~$\ref{lem:hull}$, $v$ is an internal vertex of an induced $(u,u')$-path $P$ such that $u \in (S_i - \{v\}) \cap Z \cap N(v)$ and $u' \in \langle S_i - \{v\} \rangle$, which implies that $v \in C'_i$ and $u \in C_i$. Lemma~$\ref{lem:hull}$ also implies that $\langle S_i - \{v\} \rangle - (S_i - \{v\}) \subseteq C'_i$.
	Since $u$ is adjacent to all vertices of $C'_i$, we conclude that $u' \in C_j$ for some $j \ne i$. Since $u'$ is a simplicial vertex and $u' \in \langle S_i - \{v\} \rangle$, we have that $u' \in (S_i - \{v\})$. Therefore, $P$ has exactly 3 vertices and $vu' \in E(G)$, which is a contradiction, because $v$ has no neighbors belonging to $C_j \cap S_i$ for $j \ne i$.
	
	Now, let $S$ be a maximum m-convexly independent set of $G$. We will show that $|S| \le |S_i|$ for some $i \in \{0,1, \ldots, t\}$. Write $S_Z = S \cap Z$ and $S_{\overline{Z}} = S - S_Z = S \cap C'_0$.
	First, consider that there is no edge $uv$ such that $u \in S_{\overline{Z}}$ and $v \in S_Z$. Hence, either $S_{\overline{Z}} = \emptyset$ or $S_Z \cap C_0 = \emptyset$. In the former case, since $Y_0 \cup C_0 = Z$, we have that $|S| \le |Y_0| + |C_0|$.
	Since $|Y_0| \le |Y_0| - |N(Y_0) \cap C'_0| + |C'_0|$, we have that
	$|S| \le |Y_0| + |C_0| \le |Y_0| - |N(Y_0) \cap C'_0| + |C'_0| + |C_0| \le d_c(G_0) + |X_0| = |S_0|$.
	In the latter case, we have that $|S_{\overline{Z}}| + |N(S_Z) \cap C'_0| \le |C'_0|$. Since $S_Z$ is an independent set of $G_0$, it holds that $|S_Z| - |N(S_Z) \cap C'_0| \le d_c(G_0)$. Therefore, in this case, $|S| = |S_{\overline{Z}}| + |S_Z| \le |S_{\overline{Z}}| + |N(S_Z) \cap C'_0| + d_c(G_0) \le |C'_0| + d_c(G_0) \le |X_0| + d_c(G_0) = |S_0|$.
	
	Now, consider that there is an edge $uv_i$ such that $u \in S_{\overline{Z}}$ and $v_i \in S_Z \cap C_i$ for some $i \in \{0,1, \ldots, t\}$. We claim that if there is another edge $u'v'$ such that $u' \in S_{\overline{Z}}$ and $v' \in S_Z \cap C_j$, then $i = j$. Suppose the contrary. First, consider that $u, v_i$ and $v'$ can be chosen such that $u \in N(v_i) \cap N(v')$. Then, $S$ is not an m-convexly independent set of $G$ because of the induced path $v_iuv'$. Hence, we can assume that there is no $u'' \in S_{\overline{Z}}$ such that $u'' \in N(v_i) \cap N(v')$. Now, note that $v_iuu'v'$ is an induced path of $G$, which implies that $S$ is not m-convexly independent. Since we reached a contradiction, the claim does hold.
	
	Next, we prove that for $j \ne i$ and $v_j \in S_Z \cap C_j$, it holds that $N(v_j) \cap C'_0 \subseteq C'_i$. Supposing the contrary, there is $v' \in S_Z \cap C_j$ for $j \ne i$ such that there is an edge $u'v'$ where $u' \in C'_0 - C'_i$. By the claim, we know that $v'u \not\in E(G)$. Thus, we have that $v'u'uv_i$ is an induced path of $G$ containing $u$, which is a contradiction. Hence, $N(v) \cap C'_0 \subseteq C'_i$ for every vertex $v \in S_Z$.
	
	If there was a vertex $w \in C'_0 - C'_i$ belonging to $S$, then $G$ would have the induced path $wuv_i$ and $S$ would not be m-convexly independent. Therefore, from the above, we conclude that $S - C_i$ is contained in $V(G_i)$.
	
	It remains to show that $|S| \le |S_i|$ in this case as well. Since $S_Z - C_i$ is an independent set of $G_i$, it holds that $|S_Z - C_i| - |N(S_Z - C_i) \cap C'_i| \le d_c(G_i)$.
	Since there is no edge with one extreme in $S_{\overline{Z}}$ and the other in $S_Z - C_i$, we have that 
	$|S_{\overline{Z}}| + |N(S_Z - C_i) \cap C'_i| \le |C'_i|$.
	Therefore,
	
	\[|S| = |S_{\overline{Z}}| + |S_Z| =\]
	
	\[|S_{\overline{Z}}| + |S_Z - C_i| + |S_Z \cap C_i| \le\]
	
	\[|S_{\overline{Z}}| + |N(S_Z - C_i) \cap C'_i| + d_c(G_i) + |S_Z \cap C_i| \le\]
	
	\[|C'_i| + d_c(G_i) + |S_Z \cap C_i| \le\]
	
	\[ |C'_i| + d_c(G_i) + |C_i| = |X_i| + d_c(G_i) = |S_i|.\]
	
\end{proof}

\begin{corollary}\label{cor:starlike}
	The {\sc Monophonic rank} problem restricted to starlike graphs belongs to $\P$. 
\end{corollary}

\begin{proof}
	By Theorem~\ref{the:charac}, for computing the monophonic rank of a starlike graph $G$, it suffices to compute the critical independence difference of a linear number of subgraphs of $G$. Since such subgraphs can be found in polynomial time~(\cite{PKHHT2000}) and the critical independence difference can be computed in polynomial time for general graphs~(\cite{A-1994}), the result thus hold.
\end{proof}

\nocite{*}
\bibliographystyle{abbrvnat}
\bibliography{mrank}

\begin{thebibliography}{5}
\providecommand{\natexlab}[1]{#1}
\providecommand{\url}[1]{\texttt{#1}}
\expandafter\ifx\csname urlstyle\endcsname\relax
  \providecommand{\doi}[1]{doi: #1}\else
  \providecommand{\doi}{doi: \begingroup \urlstyle{rm}\Url}\fi

\bibitem[Ageev(1994)]{A-1994}
A.~A. Ageev.
\newblock On finding critical independent and vertex sets.
\newblock \emph{SIAM Journal on Discrete Mathematics}, 7\penalty0 (2):\penalty0
  293--295, 1994.
\newblock \doi{10.1137/S0895480191217569}.

\bibitem[Cerioli and Szwarcfiter(2006)]{Cerioli2006}
M.~R. Cerioli and J.~L. Szwarcfiter.
\newblock Characterizing intersection graphs of substars of a star.
\newblock \emph{Ars Combinatoria}, 79:\penalty0 21--31, 2006.

\bibitem[Dourado et~al.(2022)Dourado, Ponciano, and {da
  Silva}]{DPS-2022-mono-rank}
M.~C. Dourado, V.~S. Ponciano, and R.~L.~O. {da Silva}.
\newblock On the monophonic rank of a graph.
\newblock \emph{Discrete Mathematics \& Theoretical Computer Science},
  24\penalty0 (2), 2022.
\newblock \doi{10.46298/dmtcs.6835}.

\bibitem[Gustedt(1993)]{Gustedt:1993}
J.~Gustedt.
\newblock On the pathwidth of chordal graphs.
\newblock \emph{Discrete Applied Mathematics}, 45\penalty0 (3):\penalty0
  233--248, 1993.

\bibitem[Peng et~al.(2000)Peng, Tang, Ko, Ho, and Hsu]{PKHHT2000}
S.-L. Peng, C.~Y. Tang, M.-T. Ko, C.-W. Ho, and T.-s. Hsu.
\newblock Graph searching on some subclasses of chordal graphs.
\newblock \emph{Algorithmica}, 27:\penalty0 395--426, 2000.

\end{thebibliography}
\label{sec:biblio}

\end{document}